\documentclass{article}

\usepackage[utf8]{inputenc} 
\usepackage{amsfonts}
\usepackage{amsmath}
\usepackage{amsthm} 
\usepackage{amssymb}
\usepackage{color}
\usepackage{graphicx}
\usepackage{setspace}
\usepackage{enumerate}
\usepackage{booktabs}
\usepackage{pstricks}
\usepackage{url}
\usepackage{hyperref}
\usepackage{mathtools}
\usepackage{cancel}
\usepackage[ruled, boxed, noend, noline, linesnumbered, norelsize]{algorithm2e}
\usepackage[normalem]{ulem}
\usepackage{multicol}
\hypersetup{
	colorlinks,
	linkcolor={red!50!black},
	citecolor={blue!50!black},
	urlcolor={blue!80!black}
}
\newtheorem{definition}{Definition}
\newtheorem{lemma}[definition]{Lemma}

\newtheorem{example}[definition]{Example}

\newtheorem{corollary}[definition]{Corollary}

\newtheorem{theorem}[definition]{Theorem}
\newtheorem{assumption}{Assumption}

\def\eqref#1{$(\ref{#1})$}
\def\NORM#1{\| #1 \|}
\def\SET#1#2{\left\{\,#1\,\left|\,#2\,\right.\right\}}
\def\SSET#1#2{\{\,#1\,|\,#2\,\}} 

\def\R{\mathbb{R}}

\def\ARGMAX{\mathop{\rm argmax}}

\def\FRAC#1#2{{\frac{\strut\displaystyle #1}{\strut\displaystyle #2}}}

\def\PROBLEM#1{\ifmmode{\langle #1 \rangle}\else\ifinner{\langle #1 \rangle}
        \else{$\langle #1 \rangle$}\fi\fi}
\def\IM(#1){{\rm Im}(#1)}

\def\INT(#1){{\rm Int}(#1)}
\def\calA{\mathcal{A}}

\def\calF{\mathcal{F}}

\def\calK{\mathcal{K}}
\def\calL{\mathcal{L}}
\def\calY{\mathcal{Y}}

\def\mathbfit#1{\mbox{\mathversion{bold}$#1$\mathversion{normal}}}
\def\b0{\mathbfit{0}} 
\def\ba{\mathbfit{a}}

\def\be{\mathbfit{e}}

\def\bu{\mathbfit{u}}
\def\bv{\mathbfit{v}}

\def\bx{\mathbfit{x}}
\def\by{\mathbfit{y}}
\def\bz{\mathbfit{z}}
\def\bxi{\mathbfit{\bxi}}

\def\bgamma{{\mathbfit{\gamma}}}

\def\b0{{\mathbfit{0}}}

\newcounter{SSS}
\begin{document}

\title{
An oracle-based projection and rescaling algorithm\\
for linear semi-infinite feasibility problems and \\
its application to SDP and SOCP
}

\author{
 Masakazu Muramatsu\thanks{
 	Department of Computer and Network Engineering, The University of Electro-Communications 1-5-1 Chofugaoka, Chofu-shi, Tokyo, 182-8585 Japan. (E-mail: \mbox{MasakazuMuramatsu@uec.ac.jp})
 }
 \and 
 Tomonari Kitahara\thanks{Kyushu University (E-mail: \mbox{tomonari.kitahara@econ.kyushu-u.ac.jp})} 
 \and
 Bruno F. Louren\c{c}o\thanks{Department of Mathematical Informatics, Graduate School of Information Science \& Technology, University of Tokyo, 7-3-1 Hongo, Bunkyo-ku, Tokyo 113-8656, Japan. (E-mail: \mbox{lourenco@mist.i.u-tokyo.ac.jp})
 	}
 \and
 Takayuki Okuno\thanks{Riken Center for Advanced Intelligence Project (E-mail: \mbox{takayuki.okuno.ks@riken.jp})} 
 \and
 Takashi Tsuchiya\thanks{
 	National Graduate Institute for Policy Studies 7-22-1 Roppongi, Minato-ku, Tokyo 106-8677, Japan. (E-mail: \mbox{tsuchiya@grips.ac.jp})
 }
}


\maketitle

\begin{abstract}
We point out that Chubanov's oracle-based algorithm for linear programming \cite{Chubanov2017}
can be applied almost as it is to linear semi-infinite programming (LSIP). 
In this note, we describe the details and 
prove the polynomial complexity of the algorithm based on the real computation model 
proposed by Blum, Shub and Smale (the BSS model) which is more suitable for 
floating point computation in modern computers. 
The adoption of the BBS model makes our description and analysis much 
simpler than the original one by Chubanov \cite{Chubanov2017}. 
Then we reformulate semidefinite programming (SDP)  and second-order cone programming (SOCP) into LSIP, and 
apply our algorithm to obtain new complexity results for computing interior feasible solutions of homogeneous SDP and SOCP. 
\end{abstract}

\section{Introduction} \label{sec:intro}
Let $T$ be a (possibly infinite) index set and $\calA = \SSET{\ba_t}{t\in T}\subseteq \R^m$.
In this paper, we consider the problem of finding a solution to the following homogeneous linear semi-infinite system:
\[
 \PROBLEM{D}\  \mbox{find } \by \mbox{ s.t. } \ba_t^T\by > 0 \ (t \in T).
\]
Without loss of generality, we assume that $\ba_t \not=\b0$ for every $t\in T$.
The objective of this article is to show that the recent oracle-based algorithm proposed by Chubanov \cite{Chubanov2017} can be adapted to solve \PROBLEM{D} and to work out the details of this extension.
The advantage of this more general setting is that we can solve, for instance, certain homogeneous feasibility problems over nonpolyhedral cones. 
In fact, we will apply our algorithm to semidefinite programming (SDP) and second-order cone programming (SOCP) in Section \ref{sec:app}.

The problem $\PROBLEM{D}$ is a special case of linear semi-infinite programming (LSIP) \cite{HK93,GL02,LS07} which 
minimizes or maximizes a linear objective function with infinitely many linear constraints.
When $|T|<\infty$, $\PROBLEM{D}$ becomes the linear feasibility problem that Chubanov \cite{Chubanov2017} dealt with.
We emphasize that the main novelty here is the case where $|T| = \infty$.

In this paper, we assume the existence of an oracle for $\PROBLEM{D}$ having the following input/output: 
\medskip\\ 
\textbf{Oracle}$(\by)$ \\
\qquad
\begin{tabular}{ll}
	\textbf{Input:}  & a nonzero vector $\by\in\R^m$ \\ 
	\textbf{Output:} &
	\begin{tabular}[t]{l} 
		$t\in T$ and $\ba_t$ such that $\ba_t^T\by \leq 0$, or  \\
		declare that $\forall t\in T, \ba_t^T \by > 0$.
	\end{tabular}
\end{tabular}
\medskip

\noindent
We show that Chubanov's oracle-based algorithm \cite{Chubanov2017} can be applied almost as it is
to $\PROBLEM{D}$ together with the oracle above.
The  algorithm described here either%
\begin{enumerate}[$(i)$]
	\item returns a solution of $\PROBLEM{D}$,
	\item returns a certificate that there is no solution for $\PROBLEM{D}$,
	\item declare that the maximum volume of parallelogram spanned by vectors in a certain bounded area in the feasible set is less than a positive constant $\epsilon$. (See Section \ref{sec:prelim} for details.)
\end{enumerate}
Furthermore, our algorithm calls the oracle polynomially many times, hence the running time is polynomial if the oracle is 
also polynomial.

We use the real computation model proposed by Blum, Shub and Smale \cite{BSS1989},
while Chubanov \cite{Chubanov2017} uses the standard bit computation model.
This means that the meaning of polynomial complexity is different from that used by Chubanov \cite{Chubanov2017}.
There are significant differences between the real computation model and the bit computation model. 
For example, it is known that determining whether an SDP has an optimal solution falls both in P and NP in the real computation model (Ramana \cite{Ramana95anexact}), while the status is not known in the bit computation model.

By adopting the real computation model, 
we can also avoid the discussion on bit size, which makes our analysis much more transparent. 
As a result, our algorithm can be regarded as a basic scheme to solve convex programming,
having a few common properties with the ellipsoid method \cite{KH80, NY83}. 
Namely, both algorithms work in the variable space
and their complexity can be written in terms of an separation oracle. 
We will discuss on this point in Section \ref{sec:concl}.

For applications of our method to convex programming, 
in the latter part of this paper, we consider to solve feasibility problems for SDP and SOCP 
using the proposed algorithm. 
We reformulate them into the shape of $\PROBLEM{D}$, and 
show that oracles needed are polynomial in the real computation model.
Therefore, polynomial algorithms based on projection and rescaling
for computing an interior-feasible solution of homogeneous SDP or SOCP are established.
We will compare the complexity of the algorithm developed here with some recent projection and rescaling algorithms \cite{PS16,LKMT17}. 

From the point of view of semi-infinite programming, a 
novel point here is that we do not assume any topology on $T$.
This stands in constrast to the usual setting
in semi-infinite programming, where it is usually assumed that $T$ is compact and/or $\ba_t$ is continuous 
under some topology.

%


\subsection{Previous works on projection and rescaling algorithms}
The algorithm described here has its origins in a previous work by 
Chubanov \cite{Ch15}, which proposed a polynomial algorithm 
to compute an interior feasible solution of homogeneous linear programming problems.
Chubanov's original algorithm was extended and improved along several different directions \cite{LRT15, Ro16,PS16,KT16,LKMT17,Fu17,PS18,ZR18,Gu18}.
Nevertheless, a common point among most of those variants is that they are divided 
in two parts, a \emph{basic procedure} and a  \emph{main algorithm}.
The basic procedure searches for an interior feasible solution,
and when it cannot find one, returns a solution called \textit{cut generating vector (CGV)}. 
The main algorithm then uses the cut generating vector to rescale the problem
so that the rescaled problem can be dealt with by the basic procedure again. 
The idea is that the rescaling makes the problem easier, so it becomes more likely that the basic procedure will find an interior feasible solution. 
Typically, the basic procedure will also make use of the projection operator.
For example, in many variants, it is necessary to orthogonally project a point onto an appropriate linear subspace.
Due to this combination of projection and rescaling steps, these algorithms are sometimes called \emph{projection and rescaling algorithms}. 

We will now discuss some of the extensions and enhancements. 
Roos \cite{Ro16} proposed an improved basic procedure which generates sharper cuts and also produced some preliminary numerical results. Improved basic procedures were also proposed by Zhang and Roos \cite{ZR18} and by Gutman \cite{Gu18}.
Li, Roos and Terlaky  \cite{LRT15} extended Chubanov's algorithm to the case of linear feasibility problems with a single non-homogeneous equality.
Pe\~{n}a and Soheili \cite{PS16} 
proposed an algorithm which computes an interior feasible solution of
homogeneous symmetric cone system using projection and rescaling. 
An extension to homogeneous feasibility problems over second order cones was proposed by Kitahara and Tsuchiya in \cite{KT16}. 
Then, Louren\c{c}o, Kitahara, Muramatsu, and Tsuchiya \cite{LKMT17}
extended Chubanov's algorithm \cite{Ch15} to symmetric cone programming in a different way than Pe\~{n}a and Soheili's algorithm and 
gave a different complexity result. 
A summary of the differences between both approaches can be found in Section 1 of \cite{LKMT17}.
Recently, Pe\~{n}a and Soheili discussed  computational aspects of projection and rescaling algorithms for the linear programming case and they presented an extensive set of numerical experiments \cite{PS18}. 

In this paper, we take a closer look at another algorithm proposed by Chubanov \cite{Chubanov2017}, which solves \PROBLEM{D} for the case $|T| < \infty$. The algorithm in \cite{Chubanov2017} is quite similar to the one in \cite{Ch15},  but an important distinction is that it makes use of an oracle.   
The advantage of using an oracle is that the algorithm works only in the space of the variable $\by$ and, under appropriate assumptions, has polynomial complexity even if, say, the number of inequalities is exponential in $m$.
This aspect of the algorithm can be quite useful and, in fact, it was 
used by Fujishige \cite{Fu17} to propose an algorithm for minimization of submodular functions. 

In Section 3.2 of Dadush et al.\cite{Dadush2016},
the same problem $\PROBLEM{D}$ equipped with the same oracle was dealt with, 
and a variant of projection and rescaling algorithms was proposed. 
Their algorithm and the algorithm proposed in this paper share the same idea to apply a projection and rescaling algorithm
to linear semi-infinite programming. However, as the directions of branch are different, there exist some differences. 
Below we describe them.

The first difference is the condition measures on which the algorithms rely.
The algorithm in \cite{Dadush2016} is based on Goffin's measure of a full-dimensional cone $\Sigma$ defined by
\[
 \rho_\Sigma = \sup\SET{r}{B(\bx, r)\subseteq \Sigma, \NORM{\bx} = 1},
\]
where $B(\bx, r)$ is the open ball whose center is $\bx$ and radius is $r$. 
In contrast, the analysis of our algorithm is based on the maximum volume of parallelograms 
spanned by vectors contained in a bounded area of the cone. See Section \ref{sec:prelim} for the exact definition of this condition measure.
For example, if $\Sigma$ is the positive orthant of the $m$ dimensional space, then the maximum volume of parallelograms is $1$,
while $\rho_\Sigma =m^{-1/2}$. On the other hand, when $\Sigma$ is the second-order cone in the three dimensional space, 
the former is $\sqrt{6}/8$ while the latter $1/\sqrt{2}$. In this case, the former is less than the latter. 
To the extent of the authors' knowledge, any useful connection between the two condition measures is not known. 

The two algorithms are different in complexity, too. 
Table \ref{tab:1} shows them.
Here, BP and MA are abbreviations of the basic procedure and the main algorithm, respectively.
The algorithm in \cite{Dadush2016} 
needs $O(m^2)$ iterations of their basic procedure and $O(m\log \rho_\Sigma^{-1})$ rescaling process, 
whereas our algorithm needs $O(m^3)$ iterations of the basic procedure and $O(\log\epsilon)$ rescaling process. 
In case of $\epsilon = \rho_\Sigma^{-1}$, the orders of total arithmetic operations needed become identical. 

\begin{table}[hpt]
\begin{center}
\begin{tabular}{l|c|c} 
 & Dadush et al. \cite{Dadush2016} & proposed \\ \hline
bd. for length of CGVs  & $O(m^{-1})$ & $O(m^{-3/2})$ \\
\# of Iterations within BP & $O(m^2)$ & $O(m^3)$ \\
\# of arith. op. within BP & $O((C_o+m^2)m^2)$ & $O((C_o+m^2)m^3)$ \\
\# of rescaling & $m\log \rho_\Sigma^{-1}$ & $\log \epsilon^{-1}$ \\
Total complexity of arith. op.& $O((C_o+m^2)m^3\log\rho_\Sigma^{-1})$ &$ O((C_o+m^2)m^3\log\epsilon^{-1})$\\
Space complexity of BP & $O(m^3)$ & $O(m^2)$ \\
Space complexity of MA & $O(m^2)$ & $O(m\min(m, \log_2\epsilon^{-1}))$
\end{tabular}
\end{center}
\caption{Differences in complexity between \cite{Dadush2016} and the proposed algorithm } \label{tab:1}
\end{table}

As is shown in the last two rows of Table \ref{tab:1}, there exist differences in space complexity, too. 
In the basic procedure, we keep $\SSET{x_t}{t\in T_+}$ where $|T_+| \leq m+1$ and 
an $(m+1)\times(m+1)$ matrix $G$, whereas the algorithm in \cite{Dadush2016} keeps 
$O(m^2)$ variables and the same number of vectors of size $m$. 
Also in the main algorithm, the proposed algorithm has a possibility to reduce the space complexity
that is depending on the number of rescaling. See Section \ref{sec:main} for the details. 

\subsection{Organization of this paper}
The paper is organized as follows.
In Section \ref{sec:prelim}, we give some preliminary observations on LSIP.
Sections \ref{sec:BP} and \ref{sec:main} describe the basic procedure and the main algorithm,
respectively, giving detailed proofs of their complexities. 
In Section \ref{sec:app}, we reformulate the problem of computing interior feasible solutions 
of a homogeneous SDP and SOCP into LSIP, and apply our algorithm with suitable oracles, respectively.
We establish the total complexities of the algorithm including that of the oracles. We compare the result of SDP 
with some existing results by Pe\~{n}a and Soheili \cite{PS16} and Louren\c{c}o, et al \cite{LKMT17}.
Finally, we give concluding remarks in Section \ref{sec:concl}.

\section{Preliminary observations} \label{sec:prelim}
Associated with $\PROBLEM{D}$, we consider the following problem:
\[
 \PROBLEM{P} 
 \left\{\begin{array}{l} \mbox{find a finite number of positive weights 
 $\SSET{x_t\in\R_{++}}{t\in T_+}$} \\
 \mbox{where $T_+\subseteq T$ and $|T_+|<\infty$
 such that $\sum_{t\in T_+}   \ba_{t} x_{t}= \b0.$}
 \end{array}
    \right.
\]
\begin{theorem} \label{th:alternative}
 It is impossible that both $\PROBLEM{P}$ and $\PROBLEM{D}$ have feasible solutions simultaneously.
\end{theorem}
\begin{proof}
Let us assume that 
$\SSET{x_t}{t\in T_+}$ and $\by\in\R^m$ are feasible solutions of $\PROBLEM{P}$ and $\PROBLEM{D}$, respectively. 
Then we have a contradiction because
$$
 0< \sum_{t\in T_+} x_{t} (\ba_{t}^T\by) = (\sum_{t\in T_+} x_{t} \ba_{t})^T\by = 0.
$$
\end{proof}
$\PROBLEM{P}$ is known as Haar's dual of $\PROBLEM{D}$, 
and if $T$  is compact and $\ba_t$ is continuous as a function of $t$, 
then they are in fact alternatives each other; $\PROBLEM{D}$ is feasible
if and only if $\PROBLEM{P}$ is infeasible (See, e.g., Lemma 1 of \cite{LS07}).
However, in this paper, we do not assume any topology on $T$;
it is possible that both $\PROBLEM{P}$ and $\PROBLEM{D}$ be simultaneously infeasible, as seen in the next example.

\begin{example}
Let $T=(0, \pi]$ and $\ba_t = (\cos t, \sin t)^T$.
Then both $\PROBLEM{D}$ and $\PROBLEM{P}$ are infeasible.
\end{example}

We consider a bounded version of the feasible region of \PROBLEM{D}:
\[
\calF_0 = \SET{\by\in\R^m}{\ba_t^T\by > \b0 \ (t\in T), \NORM{\by}\leq 1}.
\]
Obviously, $\calF_0$ is nonempty if and only if $\PROBLEM{D}$ is feasible. 
In this paper, we always work under the following assumption.
\begin{assumption}\label{as:full}
The set $\calF_0$ is full-dimensional, i.e., 
$\mbox{dim}\,\calF_0 = m$, or equivalently, there exist linearly independent vectors 
$\by_1, \ldots, \by_m \in \calF_0$. 	
\end{assumption}
The reason we work under Assumption \ref{as:full} is that we will consider the maximum volume of parallelogram spanned by vectors in $\calF_0$.
Intuitively, if this volume is large, the problem is well-conditioned and a solution of $\PROBLEM{D}$ can easily be computed through a suitable basic procedure.
If, however, the volume is small then it is hard to compute a solution in $\calF_0$.
In the extreme case where the volume is zero, we have to search a subspace in $\R^m$ containing $\calF_0$ in its relative interior,
which is harder.  

To check the strength of the assumption, we consider the extended feasible region:
\[
 \calF_1 = \SET{\by\in \R^m}{\ba_t^T\by \geq \b0\ (t\in T), \NORM{\by}\leq 1}.
\]
We have the following lemma.
\begin{lemma}\label{lemma:fd}
If $\calF_1$ is full-dimensional, then $\PROBLEM{D}$ has feasible solutions.
\end{lemma}
\begin{proof}
Let $\by$ be an interior point of $\calF_1$.
By definition, there exists $\epsilon>0$ such that 
$\by + B(\epsilon) \subseteq \calF_1$ where $B(\epsilon)$ is the open ball
whose radius is $\epsilon$.

Suppose that there is no feasible solution for $\PROBLEM{D}$.
This implies that there exists $t\in T$ such that $\ba_t^T\by=0$.
For such $t$, let $\bu = -\epsilon\ba_t / (2\NORM{\ba_t})$, which is contained in $B(\epsilon)$.
Then we have $\ba_t^T(\by + \bu) = - \epsilon \NORM{\ba_t}/2 < 0$, which contradicts 
the fact that $\by + \bu \in \calF_1$ .
Therefore, $\by$ is a feasible solution to $\PROBLEM{D}$.
\end{proof}

Lemma \ref{lemma:fd} has the following consequence. In 
\cite{Chubanov2017}, Chubanov mentions that his oracle-based algorithm solves the following problem.
\begin{equation}\label{eq:cb}
  \mbox{find } \by \mbox{ s.t. } A^T\by \geq  0, \by \neq 0,
\end{equation}
where $A=(\ba_1,\ldots,\ba_n)$ is some $m\times n$ matrix. He also assumes that the feasible region $\SSET{\by}{A^T\by \geq  0}$
is either full-dimensional or has no nonzero solution.
From Lemma \ref{lemma:fd}, we conclude that all the nontrivial problems considered in \cite{Chubanov2017} 
have a solution $\by$ that satisfy the inequalities strictly.
That is why we consider here the problem $\PROBLEM{D}$, instead of dealing with a generalized form of \eqref{eq:cb}.

Moving on, in the next example we show that  
there is a case where $\PROBLEM{D}$ has feasible solutions but $\calF_1$ is not full-dimensional.
\begin{example}
Let $T=(0, \pi)$ and $\ba_t = (\cos t, \sin t)^T$.
Then the feasible region of $\PROBLEM{D}$ is 
$\SSET{\lambda (0, 1)^T}{\lambda > 0}$, which is one-dimensional,
and so is $\calF_1$. 
\end{example}

Given an invertible matrix $M\in \R^{m\times m}$, we define the following problem scaled by $M$:
\[
 \PROBLEM{D(M)}\  \mbox{find } \tilde{\by} \mbox{ s.t. } \ba_t^T M \tilde{\by} > 0 \ (t \in T).
\]
Notice that $\PROBLEM{D} = \PROBLEM{D(I)}$.
Also, the relationship:
\[
 \by \mbox{ is feasible for } \PROBLEM{D} 
 \Leftrightarrow 
 \tilde{\by} = M^{-1} \by \mbox{ is feasible for } \PROBLEM{D(M)}
\]
is obvious.

Now we observe some properties on maximum volume of 
parallelogram spanned by vectors in a bounded set. These relations are used in the following sections 
to prove the polynomial convergence of the proposed algorithm.

For a bounded set $\calF\subseteq\R^m$, we define:
\begin{eqnarray*}
\calY(\calF)&=& \
\SET{Y\in\R^{m\times m}}{Y = (\by_1,\ldots,\by_m), \ \by_i\in\calF\ (i=1,\ldots,m)} \\
d^\ast(\calF) &=& \sup\SET{|\det Y|}{Y\in \calY(\calF)} .
\end{eqnarray*}
For an invertible matrix $M\in \R^{m\times m}$ and a scaling factor $\delta >0$, 
we define:
\[
 \calF(M, \delta) = \SET{\tilde{\by}}{\ba_t^TM\tilde{\by} \geq \b0 \ (t\in T), \NORM{\tilde{\by}}\leq \delta}.
\]
Some observations follow.
\begin{lemma} \label{lem:calF}
\begin{enumerate}
\item If $\calF\supseteq\calF'$, then $d^\ast(\calF) \geq d^\ast(\calF')$.
\item $d^\ast(M^{-1}\calF) = |\det M^{-1}| d^\ast (\calF)$.
\item $d^\ast(\calF(M, \delta)) = \delta^m d^\ast(\calF(M, 1))$.
\end{enumerate}
\end{lemma}
\begin{proof}
The first statement is obvious. 

To prove the second statement, choose an arbitrary $\tilde{Y}\in \calY(M^{-1}\calF)$, i.e., 
\[
 \tilde{Y} = \left(M^{-1}\by_1 \ldots M^{-1}\by_m \right) \mbox{ where } \by_i\in \calF \ (i=1,\ldots,m).
\]
Then we can write $\tilde{Y} = M^{-1} Y$ where $Y = \left(\by_1 \ldots \by_m\right)$.
Therefore, we have
\begin{eqnarray*}
 d^\ast(M^{-1}\calF) &=& \sup\SET{|\det \tilde{Y}|}{\tilde{Y}\in M^{-1}\calF} \\
    &=& \sup\SET{ |\det M^{-1}| |\det Y|}{Y\in\calF} \\
    &=& |\det M^{-1}| \sup\SET{ |\det Y|}{Y\in\calF} \\
     &=&  |\det M^{-1}| d^\ast(\calF),
\end{eqnarray*}
which is the second statement.

For the final statement, first observe
\[
 \delta \tilde{\by} \in \calF(M, \delta) \Leftrightarrow \tilde{\by} \in \calF(M, 1),
\]
and thus
\[
 \delta \tilde{Y}\in \calY(\calF(M, \delta))  \Leftrightarrow \tilde{Y} \in \calY(\calF(M, 1)).
\]
From this, it follows that
\begin{eqnarray*}
 d^\ast(\calF(M, \delta)) &=& \sup\SSET{|\det (\delta \tilde{Y})|}{ \tilde{Y} \in \calY(\calF(M, 1))} \\
  &=& \sup\SSET{\delta^m |\det \tilde{Y}|}{ \tilde{Y} \in \calY(\calF(M, 1))}  \\
  &=& \delta^m d^\ast(\calF(M, 1)).
\end{eqnarray*}
\end{proof}

\section{Basic Procedure} \label{sec:BP}

The basic procedure receives the scaling matrix $M$ and a positive number $\mu$, 
and returns either a solution of $\PROBLEM{D}$ or $\PROBLEM{P}$, or
positive weights $\SSET{x_t}{t\in T_+}$ that satisfy a certain property depending on $\mu$. 

Inside the basic procedure,  we use the notation
$$
 \tilde{\ba}_t = \frac{M^T \ba_t}{\NORM{M^T\ba_t}}
$$
for $t\in T$.
The complexity to compute $\tilde{\ba}_t$ from $\ba_t$ and $M$ is $O(m^2)$ in the real computation model. 
Note that we shall never compute $\SET{\tilde{\ba}_t}{t\in T}$. Since $|T| = \infty$, this is computationally intractable. 
Instead, we compute $\tilde{\ba}_t$ only when it is needed.
More specifically, we keep positive weights $\SSET{x_t}{t\in T_+}$ where $T_+\subseteq T$ with $|T_+| \leq m+1$.
Such a set of positive weights can be implemented by using appropriate data structure such as list.
We presume that $x_t = 0$ for $t\not\in T_+$.

In each iteration, 
the basic procedure calls the oracle with $M\tilde{\bz}$ where $\tilde{\bz}=\sum_{t\in T_+} x_t \tilde{\ba}_t$.
If the oracle detects feasibility of $M\tilde{\bz}$, returns it as a solution of $\PROBLEM{D}$.
If the oracle returns a violating index $\hat{t}$, 
the basic procedure computes the minimum distance point from the origin between $\tilde{\bz}$ and $\tilde{\ba}_{\hat{t}}$,
which will replace $\tilde{\bz}$. The index $\hat{t}$ is added to $T_+$,
and if $|T_+| > m+1$, a procedure called \textit{the index eliminating procedure} is 
invoked to replace the positive weights by a new set $\SSET{x_t}{t\in T_+}$ satisfying $|T_+| \leq m+1$, 
$\sum_{t\in T_+} x_t\ba_t=\tilde{\bz}$, and $\sum_{t\in T_+} x_t=1$.

Now we describe the basic procedure step by step.
In the following description, we assume that the oracle always returns the second case,
because otherwise we  immediately obtain a solution of $\PROBLEM{D}$.
 
\vspace{1cm}

\bigskip

\noindent\textbf{Basic Procedure} \\[2mm]
\qquad
\begin{tabular}{ll}
\textbf{Input:}  & a positive number $\mu>0$ and an invertible matrix $M\in\R^{m\times m}$ \\
\textbf{Output:} &\hspace{-1cm}
\begin{minipage}[t]{11cm}
\begin{enumerate}
\item  $\tilde{\by}$ such that $\forall t\in T, \ba_t^T M \tilde{\by} > 0$, or 
\item $T_+ \subseteq T$ where $|T_+| \leq m+1$, and 
  positive weights $\SSET{x_t}{t\in T_+}$  
  such that $\sum_{t\in T_+} x_t \tilde{\ba}_t = \b0$, or 
\item $T_+\subseteq T$ where $|T_+| \leq m+1$, and 
  positive weights $\SSET{x_t}{t\in T_+}$  
   such that $\NORM{\sum_{t\in T_+}x_t\tilde{\ba}_t} \leq \mu / (m+1)$.
\end{enumerate}
\end{minipage}\\
\textbf{Steps:} & 
\end{tabular}
\begin{itemize}
\item[] // \textbf{Initialization}: We suppose that $x_t = 0$ for every $t \in T$.
\item[] Step 1. Let $\tilde{\by}$ be an arbitrary nonzero vector. 
\item[] Step 2.  Let $\bar{t}\in T$ be the index returned by Oracle$(M\tilde{\by})$
and compute $\tilde{\ba}_{\bar{t}}$.
\item[] Step 3.  Let $T_+ =  \{\bar{t}\}$, $x_{\bar{t}} = 1$, and $\tilde{\bz}=\tilde{\ba}_{\bar{t}}$.  
\item[] // \textbf{Loop} 
\item[] Step 4.  Let $\hat{t}\in T$ be the index returned by Oracle$(M\tilde{\bz})$
and compute $\tilde{\ba}_{\hat{t}}$.
\item[] Step 5.  Let $T'_+=T_+\cup\{\hat{t}\}$.
\item[] Step 6.  For $t\in T'_+$, set
\begin{equation} \label{eq:BPchoiceAlpha}
 x'_{t} = \left\{\begin{array}{ll}
     \alpha x_{t} & \mbox{if } t \not=\hat{t} \\
     \alpha x_{\hat{t}} + 1-\alpha & \mbox{if } t =\hat{t} 
     \end{array} \right.
     \mbox{ where }
     \alpha = \FRAC{\tilde{\ba}_{\hat{t}}^T(\tilde{\ba}_{\hat{t}}-\tilde{\bz})}{\NORM{\tilde{\ba}_{\hat{t}}-\tilde{\bz}}^2}.
\end{equation}
\item[] Step 7.  Let $\tilde{\bz} = \sum_{t\in T'_+} x'_t\tilde{\ba}_t$.
\item[] Step 8. Call \textit{index elimination procedure} to compute $T_+$ and $\SSET{x_t}{t \in T_+}$ 
such that $|T_+| \leq m$, $x_t \geq 0 (t\in T_+)$, $\sum_{t\in T_+}x_t = 1$, and   
$\sum_{t\in T_+} x_t \tilde{\ba}_t = \tilde{\bz}$.
\item[] Step 9. If $\tilde{\bz} = \b0$, then return $\SSET{x_t}{t\in T_+}$ with the second status. 
\item[] Step 10.  If $\NORM{\tilde{\bz}}\leq \mu/(m+1)$, then return the third status with $\SET{x_t}{t\in T_+}$.
\item[] Step 11.  Go to Step 4.
\end{itemize}
\medskip

The index elimination procedure works as follows. 
\begin{enumerate}
\item During the first $m$ iterations, it does nothing. (Just set $T_+=T'_+$ and $x_t = x'_t$ for every $t\in T_+$.)
\item At the $(m+1)$-th iteration, it computes\footnote{If this matrix is not invertible, then just skip this procedure until 
we have $m+1$ independent vectors.} 
\begin{equation} \label{eq:G}
G =  \left(\begin{array}{ccc}\tilde{\ba}_{t_1} & \ldots & \tilde{\ba}_{t_{m+1}}\\
                1 & \ldots &  1 \end{array} \right)^{-1},
\end{equation}
where $T_+ = T'_+ = \{t_1, \ldots, t_{m+1}\}$, and store it. 
\item If $|T'_+| = m+2$ with $\hat{t}$ being the last index added, then compute:
\[
 \bgamma = G \left(\begin{array}{c}\ba_{\hat{t}}\\ 0 \end{array} \right).
\]
Let 
\[
 t^\ast = \ARGMAX\SET{\beta}{x'_t-\beta\gamma_t\geq 0}
   = \ARGMAX\SET{-x'_t/\gamma_t}{\gamma_t < 0}
\]
and $\beta^\ast = -x'_{t^\ast}/\gamma_{t^\ast}$.
If $\beta^\ast \geq x'_{\hat{t}}$, then set
\[
 x_t = x'_t+x'_{\hat{t}}\gamma_t \ (t\in T_+).
\]
Else, set
\[
 x_t = x'_t - \frac{x'_{t^\ast}}{\gamma_{t^\ast}}\gamma_t \ (t\in T'_+\backslash \{t^\ast, \hat{t}\}),\ 
 x_{\hat{t}} = x'_t + \frac{x'_{t^\ast}}{\gamma_{t^\ast}},
\]
and update
\[
 G \leftarrow G - \FRAC{G\be_{t^\ast}\bar{\ba}^TG}{1+\be_{t^\ast}^TG\bar{\ba}}, \ 
 T_+ \leftarrow T'_+ \backslash \{t^\ast\}, 
\]
where $\bar{\ba} = \left(\begin{array}{c} \ba_{\hat{t}}-\ba_{t^\ast} \\ 0 \end{array}\right)$
and  $\be_{t^\ast}$ is the vector having $1$ at the position corresponding to $t^\ast$ and zero otherwise. 
\end{enumerate}

Using the index elimination procedure, we can bound the size of $T_+$ by $m+1$.
\begin{lemma}
The index elimination procedure returns $T_+$ and $\SSET{x_t}{t\in T_+}$ such that 
\begin{eqnarray}
x_t&\geq&  0 \ (t\in T_+) \label{eq:lemIEP1} \\
\sum_{t\in T_+} x_t \tilde{\ba}_t &=&\tilde{\bz}  \label{eq:lemIEP2}  \\
\sum_{t\in T_+} x_t &=& 1 \label{eq:lemIEP3} \\
|T_+| &=& m+1. \label{eq:lemIEP4}
\end{eqnarray}
Furthermore, \eqref{eq:G} holds with a suitable arrangement of vectors $\SSET{\tilde{\ba}_t}{t\in T_+}$
when the index elimination procedure ends.
\end{lemma}
\begin{proof}
In the procedure, there are two cases, (i) $\beta^\ast \geq x'_{\hat{t}}$, and (ii) $\beta^\ast < x'_{\hat{t}}$,
and we divide the proof accordingly. 

\noindent\textbf{Case (i).}
When $\beta^\ast\geq x'_{\hat{t}}$, 
since $T_+$ does not change, \eqref{eq:G} and \eqref{eq:lemIEP4} hold automatically. 
Since
\[
 \tilde{\bz} = \sum_{t\in T'_+}x'_t\tilde{\ba}_t
  = \sum_{t\in T_+\backslash\{\hat{t}\}} \left(x'_t + x'_{\hat{t}}\gamma_t\right) \tilde{\ba}_t
  = \sum_{t\in T_+\backslash\{\hat{t}\}} x_t \tilde{\ba}_t,
\]
\eqref{eq:lemIEP2} also holds. 
The equality $\sum_{t\in T_+\backslash\{\hat{t}\}} \gamma_t = 0$ implies \eqref{eq:lemIEP3}. 

\noindent\textbf{Case (ii).}
When $\beta^\ast < x'_{\hat{t}}$,
since we replace $t^\ast$ with $\hat{t}$, 
it is obvious that \eqref{eq:lemIEP4} holds. 

The equation \eqref{eq:G} implies
\[
 \sum_{t\in T'_+\backslash\{\hat{t}\}} \gamma_t\tilde{\ba}_t = \ba_{\hat{t}},
\] 
from which it follows that
\[
 \ba_{t^\ast} = -\sum_{t\in T_+\backslash\{\hat{t}\}} \gamma_t/\gamma_{t^\ast}\tilde{\ba}_t + 1/\gamma_{t^\ast} \tilde{\ba}_{\hat{t}}.
\]
Now we have
\begin{eqnarray*}
 \sum_{t\in T_+} x_t \tilde{\ba}_t 
 &=&  \sum_{t\in T_+\backslash\{\hat{t}\}} \left(x'_t -\frac{x'_{t^\ast}}{\gamma_{t^\ast}}\gamma_t\right)\tilde{\ba}_t + 
   \left(x'_{\hat{t}}+\frac{x'_{t^\ast}}{\gamma_{t^\ast}}\right) \tilde{\ba}_{\hat{t}} \\
 &=&  \sum_{t\in T'_+} x'_t\tilde{\ba}_t = \tilde{\bz}.
\end{eqnarray*}
Similarly, \eqref{eq:lemIEP3} easily follows. 

Due to the definition of $t^\ast$, $x_t \geq 0$ for every $t\in T_+\backslash\{\hat{t}\}$,
and for $\hat{t}$, $x_{\hat{t}}> 0$ since $\beta^\ast<x'_{\hat{t}}$. 

Finally, the Sherman-Morrison-Woodbury formula can be applied to the relation
\[
  \left(\begin{array}{ccccc}\tilde{\ba}_{t_1} & \ldots & \tilde{\ba}_{\hat{t}}& \ldots & \tilde{\ba}_{t_{m+1}}\\
                1 &\ldots & 1 & \ldots & 1 \end{array} \right)
 =  \left(\begin{array}{ccccc}\tilde{\ba}_{t_1} & \ldots & \tilde{\ba}_{t^\ast} & \ldots & \tilde{\ba}_{t_{m+1}}\\
                1 &\ldots & 1 & \ldots &  1 \end{array} \right)
                + \bar{\ba}(\be_{t^\ast})^T,
\]
which produces
\[
  \left(\begin{array}{ccccc}\tilde{\ba}_{t_1} & \ldots & \tilde{\ba}_{\hat{t}}& \ldots & \tilde{\ba}_{t_{m+1}}\\
                1 & \ldots &1 & \ldots &  1 \end{array} \right)^{-1}
                = \FRAC{G\be_{t^\ast}\bar{\ba}^TG}{1+\be_{t^\ast}^TG\bar{\ba}}.
\]
This implies that the new $G$ is the inverse of the left-hand side of the above,
which proves \eqref{eq:G}.
\end{proof}

The choice of $\alpha$ in Step 6 ensures the following property of the procedure.
\begin{lemma} \label{lem:BPineq}
Suppose that $\sum_{t\in T_+} x_t = 1$ at the beginning of Step 6. 
Then it hold that
\begin{equation} \label{eq:BPeq1}
 \sum_{t\in T'_+}x'_t = 1
\end{equation}
and 
\begin{equation} \label{eq:BPineq1}
 \FRAC{1}{\NORM{\sum_{t\in T'_+}\tilde{\ba}_t x'_t}^2} \geq  \FRAC{1}{\NORM{\sum_{t\in T_+}\tilde{\ba}_t x_t}^2} +  1
\end{equation}
at the end of Step 6.
\end{lemma}
Lemma \ref{lem:BPineq} can be proved by the standard arguments. For completeness, we give one in Appendix A.
\begin{theorem}
The basic procedure stops in at most $(m+1)^2/\mu^2$ iterations.
If the complexity of Oracle is $C_o$, then the total complexity of Basic procedure is $O((m^2+C_o) m^2/\mu^2)$.
\end{theorem}
\begin{proof}
The algorithm stops at Step 10 when 
\begin{equation} \label{eq:stop9}
 \frac{1}{\NORM{\sum_{t\in T_+} \tilde{\ba}_t \bx'_t}^2} \geq \frac{(m+1)^2}{\mu^2}.
\end{equation}
Since, starting from a positive value, the left-hand side is increased at least by $1$ at each iterations, 
\eqref{eq:stop9} is met at most $(m+1)^2/\mu^2$ iterations.

Now we check the complexity of one iteration of the basic procedure. 

Step 4 is $O(C_o+m^2)$ where $C_o$ is the complexity of the oracle, 
since we have to compute $M\ba_{\hat{t}}$ for $\tilde{\ba}_{\hat{t}}$.
Step 6 is $O(m)$. 
Step 7 needs $O(m)$, since the new $\tilde{\bz}$ can be computed by 
$\alpha\tilde{\bz} + (1-\alpha) \tilde{\ba}_{\hat{t}}$.

In Step 8, we call the index elimination procedure. 
At the $(m+1)$-th calls of the index elimination procedure, we need to compute the 
inverse of an $(m+1)\times (m+1)$ matrix, which costs $O(m^3)$, but this is only once. 
Each call after that includes only matrix-vector computations, and needs $O(m^2)$. 

Summing up them all, the total complexity is $O((m^2+C_o) m^2/\mu^2)$.
\end{proof}

Let $\bx'$ be the vector returned by the basic procedure at Step 10. 
Since $\sum_{t\in T_+} x'_t = 1$, $x'_t\geq 0\ (t\in T_+)$ and $|T_+| \leq m+1$, 
we have 
\[
 \frac{1}{m+1} \leq \max_{t\in T_+} x'_t \leq 1.
\]

\begin{lemma} \label{lem:BP}
Let $\tilde{t} = \ARGMAX\SSET{x'_t}{t\in T_+}$. 
Then for arbitrary $\by$ such that $\tilde{\ba}_t^T\by > 0\ (t\in T)$, it holds that
\[
 0< \tilde{\ba}_{\tilde{t}}^T\by  \leq \mu\NORM{\by}.
\]
\end{lemma}
\begin{proof}
We have 
\[
0\leq \frac{1}{m+1}\tilde{\ba}_{\tilde{t}}^T\by \leq x'_{\tilde{t}} \tilde{\ba}_{\tilde{t}}^T\by
  \leq \sum_{t\in T_+} x'_t \tilde{\ba}_t^T\by
  \leq \NORM{\sum_{t\in T_+}^n x'_t \tilde{\ba}_t}\NORM{\by} 
  \leq \frac{\mu}{m+1}\NORM{\by}.
\]
\end{proof}

\section{The main algorithm} \label{sec:main}

The main algorithm receives an instance of $\PROBLEM{D}$ and a positive number $\epsilon$.
The main algorithm calls the basic procedure with a suitable scaling matrix, and if it returns a solution of 
$\PROBLEM{D}$ or $\PROBLEM{P}$, then it finishes with the solution. 
Otherwise, it uses positive weights that is returned by the basic procedure
to make a scaling matrix, and call the basic procedure again. 
The number of calls of the basic procedure from the main algorithm is bounded by 
\[
 s^\ast_\epsilon = \left(\frac{1}{2} \log_2 e\right)\log_2\epsilon^{-1}, 
\]
as we will see in the following.

Below we describe the main algorithm. 

\medskip

\noindent\textbf{Main Algorithm}\\[2mm]
\qquad
\begin{tabular}{ll}
\textbf{Input:}  & an instance of $\PROBLEM{D}$ and a positive number $\epsilon$ \\
\textbf{Output:} &\hspace{-1cm}
\begin{minipage}[t]{11cm}
\begin{enumerate}
\item  $\by$ such that $\forall t\in T, \ba_t^T \by > 0$, or 
\item $T_+ \subseteq T$ where $|T_+| \leq m+1$, and 
  positive weights $\SSET{x_t}{t\in T_+}$  
  such that $\sum_{t\in T_+} x_t \ba_t = \b0$, or 
\item declare that $d^\ast(\calF(I, 1))\leq \epsilon$.
\end{enumerate}
\end{minipage}\\
\textbf{Steps:} & 
\end{tabular}
\begin{itemize}
\item[] Step 1. Let $s=1$ and $M_s=I$.
\item[] Step 2.  If $s\geq s^\ast_\epsilon$, then stop. $d^\ast(\calF(I, 1))\leq \epsilon$.
\item[] Step 3.  Call Basic Procedure with $M_s$ and $\mu = \sqrt{3m}^{-1}$. \\
\quad --- If a solution of $\PROBLEM{D}$  or $\PROBLEM{P}$ is returned, then return it.  \\
\quad --- Otherwise, let $\SSET{x_t}{t\in T_+}$ be the returned positive weights. 
\item[] Step 4.  Let $\tilde{t}=\ARGMAX\SSET{x_t}{t\in T_+}$.
\item[] Step 5.  Let $\tilde{\ba}_{\tilde{t}} = M_s^T\ba_{\tilde{t}}/\NORM{M_s^T\ba_{\tilde{t}}}$ and 
\begin{equation} 
 D_{\tilde{t}} = I - \frac{1}{2}\tilde{\ba}_{\tilde{t}}\tilde{\ba}_{\tilde{t}}^T.  \label{eq:Dk}
\end{equation}
\item[] Step 6.  Let $M_{s+1} = M_s D_{\tilde{t}}$.
\item[] Step 7. Let $s=s+1$ and go to Step 2.
\end{itemize}

\medskip

Now we analyze the complexity of the main algorithm.
The proof is basically the same as the argument presented before Theorem 2.1  in Chubanov \cite{Chubanov2017}.
\begin{lemma} \label{lem:normbound}
Suppose that the basic procedure called with $M$ returned $\SSET{x_t}{t\in T_+}$. Let $\tilde{t}=\ARGMAX\SSET{x'_t}{t\in T_+}$
and define $D_{\tilde{t}}$ by \eqref{eq:Dk}.
For every $\tilde{\by}\in \calF(M, 1)$, it holds that 
\[
 \NORM{D_{\tilde{t}}^{-1}\tilde{\by}} \leq \NORM{\tilde{\by}}\sqrt{1+\frac{1}{m}}.
\]
\end{lemma}
\begin{proof}
Given $\tilde{\by}\in \calF(M, 1)$, we first decompose:
\[
 \tilde{\by} = \lambda \tilde{\ba}_{\tilde{t}} + \bu
\]
uniquely where $\bu\in\ker \tilde{\ba}_{\tilde{t}}^T$ and $\lambda \in \R$.
Due to Lemma \ref{lem:BP} with $\mu = \sqrt{3m}^{-1}$, we have
\[
|\lambda|\NORM{\tilde{\ba}_{\tilde{t}}} =  |\tilde{\ba}_{\tilde{t}}^T\tilde{\by}|  \leq \FRAC{\NORM{\tilde{\by}}}{\sqrt{3m}}.
\]
Since
\[
 D_{\tilde{t}}^{-1} = I + \tilde{\ba}_{\tilde{t}}\tilde{\ba}_{\tilde{t}}^T,
\]
we have
\begin{eqnarray*}
\NORM{D_{\tilde{t}}^{-1}\tilde{\by}}^2 &=& \left\|\left(I +\ba_{\tilde{t}}\ba_{\tilde{t}}^T\right) \by\right\|^2 \\
  &=& \NORM{2\lambda \tilde{\ba}_{\tilde{t}} + \bu}^2  = 4\lambda^2 \NORM{\tilde{\ba}_{\tilde{t}}}^2 + \NORM{\bu}^2 \\
  &=& 3\lambda^2\NORM{\tilde{\ba}_{\tilde{t}}}^2 + \NORM{\tilde{\by}}^2 \\
  &\leq& \NORM{\tilde{\by}}^2\left(\frac{1}{m}+1\right).
\end{eqnarray*}
\end{proof}
\begin{lemma} \label{lem:dast}
Suppose that the basic procedure called with $M$ 
returned  $\SSET{x_t}{t\in T_+}$. Let $\tilde{t}=\ARGMAX\SSET{x'_t}{t\in T_+}$ and 
define $D_{\tilde{t}}$ by \eqref{eq:Dk}.
Then it holds that
\[
 d^\ast(\calF(M, 1)) \leq \frac{\sqrt{e}}{2} d^\ast(\calF(MD_{\tilde{t}}, 1)).
\]
\end{lemma}
\begin{proof}
For arbitrary $\bu\in\calF(M, 1)$, it follows from Lemma \ref{lem:normbound} that
\[
 \NORM{D_{\tilde{t}}^{-1}\bu} \leq \sqrt{1+\frac{1}{m}} \NORM{\bu} \leq \sqrt{1+\frac{1}{m}}.
\]
This and the fact that 
\[
 \ba_t^TM D_{\tilde{t}} D_{\tilde{t}}^{-1}\bu > 0 \ (t \in T)
\]
lead $D_{\tilde{t}}^{-1}\bu \in \calF(MD_{\tilde{t}}, \sqrt{1+m^{-1}})$, 
which implies $\calF(MD_{\tilde{t}}, \sqrt{1+m^{-1}}) \supseteq D_{\tilde{t}}^{-1} \calF(M, 1)$.
Now the first statement of Lemma \ref{lem:calF} implies
\begin{equation} \label{eq:dast_ineq}
 d^\ast(\calF(MD_{\tilde{t}}, \sqrt{1+m^{-1}})) \geq d^\ast(D_{\tilde{t}}^{-1} \calF(M, 1)),
\end{equation}
and the third statement of Lemma \ref{lem:calF} leads
\[
 d^\ast(\calF(MD_{\tilde{t}}, \sqrt{1+m^{-1}})) = \left(1+\frac{1}{m}\right)^{m/2} d^\ast(\calF(MD_{\tilde{t}}, 1))
  < \sqrt{e} d^\ast(\calF(MD_{\tilde{t}}, 1)).
\]
On the other hand, using the second statement of Lemma \ref{lem:calF} and the fact 
$\det D_{\tilde{t}}^{-1}= 2$, we can rewrite the right-hand side of \eqref{eq:dast_ineq} as:
\[
d^\ast(D_{\tilde{t}}^{-1} \calF(M, 1)) = |\det D_{\tilde{t}}^{-1}|d^\ast(\calF(M, 1)) = 2 d^\ast(\calF(M, 1)).
\]
Substituting these relations into \eqref{eq:dast_ineq}, we obtain
\[
 \sqrt{e}d^\ast(\calF(MD_{\tilde{t}}, 1)) \geq 2 d^\ast(\calF(M, 1)),
\]
which is the desired relation.
\end{proof}
\begin{theorem}
If the main algorithm stops in $s^\ast_\epsilon$ calls of the basic procedure without finding any feasible solution of $\PROBLEM{D}$
or $\PROBLEM{P}$,
then $d^\ast(\calF(I, 1))\leq \epsilon$.
\end{theorem}
\begin{proof}
We assume that the $s$-th call of the basic procedure where $1\leq s \leq s^\ast_\epsilon$
returns $\SSET{x^s_t}{t\in T^s_+}$, and let $\tilde{t}_s=\ARGMAX\SSET{x^s_t}{t\in T^s_+}$.
Define $D_{\tilde{t}_s}$ by \eqref{eq:Dk}.
Then Lemma \ref{lem:dast} implies that
\begin{eqnarray*}
d^\ast(\calF(I, 1)) &\leq& \frac{\sqrt{e}}{2} d^\ast(\calF(D_{\tilde{t}_1}, 1)) \\
   &\vdots& \\
 &\leq& \left(\frac{\sqrt{e}}{2}\right)^{s^\ast_\epsilon} d^\ast(\calF(D_{\tilde{t}_1} \cdots D_{\tilde{t}_{s^\ast_\epsilon}}, 1))  \\
 &\leq& \left(\frac{\sqrt{e}}{2}\right)^{s^\ast_\epsilon}
\end{eqnarray*}
where the last inequality is due to the fact that
$d^\ast(\calF(D_{\tilde{t}_1}\cdots D_{\tilde{t}_{s^\ast_\epsilon}}, 1)) \leq 1$,
because the length of each edge of the parallelogram is bounded by $1$.

Now it is easy to see that, by the definition of $t^\ast$, it holds that 
\[
\left(\frac{\sqrt{e}}{2}\right)^{s^\ast_\epsilon} \leq \epsilon.
\]
\end{proof}

Since we call the basic procedure with $\mu = \sqrt{3m}^{-1}$, 
we have the following complexity result for the main algorithm.
\begin{corollary} \label{cor:lsip}
The main algorithm returns a feasible solution of $\PROBLEM{D}$ or $\PROBLEM{P}$,  or 
declare that $d^\ast(\calF(I, 1)) \leq \epsilon$ in 
$O((m^2+C_o)m^{3}\log\epsilon^{-1})$ arithmetic operations if $C_o$ is the complexity of the oracle.
\end{corollary}

Finally, we discuss the space complexity needed by the main algorithm. 
If $m > \log_2 \epsilon^{-1}$, it is a clever choice to store $\{\tilde{\ba}_{t_1}, \ldots, \tilde{\ba}_{t_s}\}$ instead of $M_s$,
because we can easily compute 
\[
 M_s \bz = D_{t_1}\cdots D_{t_s} \bz
\]
by using $\{\tilde{\ba}_{t_1}, \ldots, \tilde{\ba}_{t_s}\}$ in $O(ms)$ where $ s \leq \log_2 \epsilon^{-1}<m$.
Therefore, in this case, we can reduce the space complexity to $O(m\log_2\epsilon^{-1})$ without sacrificing the complexity 
of arithmetic operations. 

\section{Applications to SDP and SOCP} \label{sec:app}

In this section, we use our projection and rescaling algorithm for solving SDP and SOCP.

\subsection{SDP} \label{sec:SDP}
Let us consider the SDP feasibility problem:
\[
 \PROBLEM{LMI}
  \mbox{ find } 
  \by \mbox{ s.t. } 
  \sum_{i=1}^m y_i A_i \succ O,
\]
where $A_i (i=1,\ldots,m)$ are symmetric matrices and $A\succ O$ means that $A$ is positive definite.
This problem is particular case of what is called a \emph {linear matrix inequality (LMI)},
which appears in the theory of $H_\infty$ control \cite{boyd1994linear}.

The problem $\PROBLEM{LMI}$ can be cast into an LSIP problem as follows:
\[
 \PROBLEM{LMI/LSIP} 
  \mbox{ find } 
  \by \mbox{ s.t. }
  \sum_{i=1}^m y_i A_i \bullet (\bv\bv^T) > 0 \ (\bv \in \bar{B}),
\]
where $\bar{B}=\SSET{\bv}{\NORM{\bv}=1}$ is the surface of the unit ball whose center is at the origin.
Using the notation $(\ba_{\scriptsize\bv})_i = A_i\bullet (\bv\bv^T) = \bv^TA_i\bv$
for $\bv\in \bar{B}$, $\PROBLEM{LMI/LSIP}$ can be expressed as
\[
 \mbox{ find } \by \mbox{ s.t. } \ba_{\scriptsize\bv}^T\by > 0 \ (\bv\in \bar{B}).
\]
Given $\bv$, computing $\ba_v$ is $O(mn^2)$ if each $A_i$ is dense.

Now we apply the algorithm proposed in this paper to $\PROBLEM{LMI/LSIP}$.
The oracle for $\PROBLEM{LMI/LSIP}$ should have the following input and output.
\medskip

\noindent\textbf{SDPOracle:} \\
\ \textbf{Input: } $n\times n$ symmetric matrix $X$ \\
\ \textbf{Output:} $\bv\in \bar{B}$ such that $\bv^TX\bv \leq 0$, or declare $X$ is positive definite.

\medskip

The oracle can be implemented by, say, a modification of Cholesky factorization.
Let us consider to apply Cholesky factorization to $X$. 
If we can factorize  to the end, then we obtain $X=LL^T$ with full-ranked $L$ and we know that $X$ is positive definite.
Otherwise, we cannot continue at some point of Cholesky factorization by getting nonpositive 
pivot at a diagonal, and we realize that $X$ is not positive definite.
Then, a certificate can be constructed by using the row where the nonpositive pivot appeared.
See also \cite{Edmonds}. Overall, we can implement the oracle for $\PROBLEM{LMI/LSIP}$
so that it runs in $O(n^3)$ where $n$ is the size of the matrix $X$.

Note that as soon as the oracle returns $\bv$, we have to compute $\ba_{\bv}$, which is $O(mn^2)$. 
After all, the complexity of the oracle including the computation of $\ba_{\bv}$ is $O((m+n)n^2)$.

We state the complexity of the proposed algorithm applied for SDP. 
\begin{theorem}
Suppose that we reformulate $\PROBLEM{LMI}$ into $\PROBLEM{LMI/LSIP}$ and 
solve it by the main algorithm using the oracle above. 
Then the algorithm with SDPOracle computes a feasible solution of $\PROBLEM{LMI/LSIP}$ or declare that 
$\calF(I, 1)\leq \epsilon$ in 
\[
O\left((m^2+mn^2+n^3)m^{3}\log\epsilon^{-1}\right).
\]
\end{theorem}
\begin{proof}
By Corollary \ref{cor:lsip}, since the size of $\by$ is $m$ and $C_o = O((m+n)n^2)$, 
we immediately obtain the result. 
\end{proof}

In Table \ref{tab:comparison},  we summarize the complexity of solving SDP
by three projection and rescaling algorithms: 
Pe\~{n}a and Soheili \cite{PS16}, 
Louren\c{c}o, Kitahara, Muramatsu and Tsuchiya (LKMT) \cite{LKMT17},
and the proposed one.

\begin{table}[hpt]
\begin{center}
\begin{tabular}{c|c|c}
Pe\~{n}a \& Soheili  \cite{PS16} &LKMT \cite{LKMT17} & Proposed \\ \hline
$O(mn^6\log\delta^{-1})$ & $O((m^3+m^2n^2+n^4m)n\log\tilde{\epsilon}^{-1})$  & $O((m^2+mn^2+n^3)m^3\log\epsilon^{-1})$
\end{tabular}
\end{center}
\caption{Complexity Comparison in case of SDP} \label{tab:comparison}
\end{table}
The first entry in Table \ref{tab:comparison} corresponds
to the projection and rescaling algorithm by Pe\~{n}a and Soheili \cite{PS16}, when the von Neumann scheme is used as a basic procedure. For the von Neumann scheme, see Section 5 in \cite{PS16} for more details.
Basically, the basic procedures used in \cite{LKMT17} and the proposed algorithm are categorized 
to this scheme. 
The second entry in Table \ref{tab:comparison} corresponds to the 
bound for the algorithm developed in \cite{LKMT17}, see the remarks after Theorem 17 in \cite{LKMT17}. Here, we specialize the complexity bound obtained in \cite{LKMT17} to the case of a single positive semidefinite cone.

Our algorithm looks less dependent on $n$, compared to the others. 
However, there are a few caveats when trying to compare the complexity results described in Table \ref{tab:comparison}. 
In all three algorithms different measures are used in the complexity.
In Pe\~{n}a and Soheili \cite{PS16},
$\delta$ is the condition measure of the linear subspace and the semidefinite cone, 
and in \cite{LKMT17}, $\tilde{\epsilon}$ is an upper bound of minimum eigenvalue 
of feasible solution contained in a certain half space.
In the proposed method, $\epsilon$ is the maximum volume of parallelogram 
spanned by vectors in the bounded feasible region $\calF_0$.
Although we do not know the precise relationships between those 
measures, it is hard to imagine that there are significant order differences 
between them. 

The ability to exploit sparsity in $X$ is an advantage of our algorithm to the others.
Since our algorithm pushes the numerical part concerning $X$ into an oracle,
if we know in advance that $X$ has a sparse structure so that its positive semidefiniteness can be computed efficiently, 
we can immediately take advantage of it.
In contrast, to the best of the authors' knowledge, how to exploit sparsity in $X$
has not known for the algorithms in \cite{PS16} and \cite{LKMT17}, since they use projections in the space of $X$.


\subsection{SOCP} \label{sec:SOCP}

We denote the $k$ dimensional second-order cone by $\calK_k$, i.e., 
\[
 \calK_k = \SET{\left(\begin{array}{c} x_0 \\ \tilde{\bx}\end{array}\right)\in \R\times \R^{k-1}}{x_0 \geq \NORM{\tilde{\bx}}}.
\]
If we denote 
\[
 \calL_k =  \SSET{\bx\in\R^k}{x_0 = 1},
\]
then we can easily see the following equivalence:
\[
 \bx\in \mbox{Int} (\calK_k) \Leftrightarrow \forall \ba \in \calK_k\cap \calL_k, \ \ba^T\bx > 0.
\]
Furthermore, if $\bx = (x_0, \tilde{\bx}) \not \in \mbox{Int}(\calK_k)$ and $\tilde{\bx}\not=\b0$, then 
for $\bv = (1, -\tilde{\bx}/\NORM{\tilde{\bx}})\in\calK_k\cap \calL_k$, we have
\begin{equation} \label{eq:soccheck}
 \bv^T\bx = x_0 - \NORM{\tilde{\bx}} \leq 0.
\end{equation}
In other words, we can check whether $\bx$ is in the interior of $\calK_n$ or not, 
and if not, can find $\bv\in \calK_n\cap\calL_n$ such that $\bv^T\bx \leq 0$ in $O(n)$.

We consider the following feasibility problem of homogeneous SOCP:
\[
 \PROBLEM{SOCP}
  \mbox{ find } 
  \by \mbox{ s.t. } 
  A^T\by \in \calK
\]
where $A \in \R^{m\times n}$,  
$\calK=\calK_{n_1}\times\cdots\times \calK_{n_p}$,
and $n = \sum_{i=1}^p n_i$.
According to the structure of $\calK$, we set 
$\calL = \calL_{n_1}\times \cdots \times \calL_{n_p}$.

The aim of this subsection is to solve $\PROBLEM{SOCP}$ by using the proposed algorithm 
for LSIP. To this end, we reformulate \PROBLEM{SOCP} into 
\[
 \PROBLEM{SOCP/LSIP} 
  \mbox{ find } 
  \by \mbox{ s.t. }
  \sum_{i=1}^m y_i \ba_i^T \bv > 0 \ (\bv \in \calK\cap\calL),
\]
and what we request for inputs and outputs of the oracle are as follows. 

\noindent\textbf{SOCPOracle:} \\
\ \textbf{Input: } $\bx \in \R^n$ \\
\ \textbf{Output:} $\bv\in \calK\cap\calL$ such that $\bx^T\bv\leq 0$, or declare $\bx\in \mbox{Int}(\calK)$.

Due to the statement just after \eqref{eq:soccheck}, it is clear that 
this oracle will run in $O(n)$, where $n = \sum_{i=1}^p n_i$.
Now we obtain the complexity of computing a solution of $\PROBLEM{SOCP/LSIP}$.
\begin{theorem}
Suppose that we reformulate $\PROBLEM{SOCP}$ into $\PROBLEM{SOCP/LSIP}$ and 
solve it by the main algorithm using the oracle above. 
Then the algorithm computes a feasible solution of $\PROBLEM{SOCP/LSIP}$ or declare that 
$\calF(I, 1)\leq \epsilon$ in 
\[
O\left((m^2+n)m^{3}\log\epsilon^{-1}\right).
\]
\end{theorem}
\begin{proof}
Plug in the complexity of the SOCPOracle into $C_o$ in Corollary \ref{cor:lsip}, 
and we immediately obtain the result.
\end{proof}

\section{Concluding Remarks} \label{sec:concl}

We have extended Chubanov's recent work \cite{Chubanov2017} on a projection and rescaling algorithm
to the framework of LSIP. Chubanov's idea was applied almost word by word with significant simplification. 
Then we applied the proposed algorithm to SDP and SOCP, and showed a polynomial complexity of the algorithm.

We now discuss briefly the relation between our approach and the ellipsoid method .
Both the proposed oracle-based projection and rescaling algorithm and the ellipsoid method \cite{KH80, NY83} 
work in the variable space using an separation oracle; they can be viewed as general scheme for linear and nonlinear problems.
One may also consider applying the ellipsoid algorithm to LSIPs with the aid of an oracle. Indeed, this type of extension of the ellipsoid method was developed for SDP in the seminal work  by Gr\"{o}tschel, Lov\'{a}sz and Schrijver on combinatorial optimization \cite{GLS1,GLS2}. 
According to the textbook by Bubeck \cite{Bubeck2015}, 
the complexity to solve an SDP by using the ellipsoid method using SDPOracle 
is $O(\max(m, n)n^6\log \hat{\epsilon}^{-1})$, where $\hat{\epsilon}>0$ is
another measure for optimality. Note that this complexity is to solve an optimization problem, 
not for the homogeneous feasibility problem, but from the complexity point of view, they are the same. 

It is well-known that the ellipsoid method is not practical at all, 
because the algorithm proceeds just as the theory predicts, and cannot take any advantage of the easiness of 
real-world problems. 

Since our algorithm uses only projection and rescaling, its implementation is by far easier than that of the ellipsoid method.
However, we do not know whether the algorithm is efficient enough to solve real-world problems. 
To check the practicality of the proposed algorithm through extensive numerical experiments 
is an important next step. 

\section*{Acknowledgements}
M.~Muramatsu is supported in part by Grant-in-aid for Scientific Research(C) 17K00031.
T.~Tsuchiya, B.~Louren\c{c}o, and T.~Kitahara
are supported in part by Grant-in-aid for Scientific Research (B) 18H03206.
T.~Okuno is supported in part by Grant-in-aid for Young Scientists (B) 15K15943.
T.~Kitahara is supported in part by Grant-in-aid for Young Scientists (B) 15K15941.

\section*{Appendix A: Proof of Lemma \ref{lem:BPineq}}

The equality \eqref{eq:BPeq1} can  easily be seen by 
\[
 \sum_{t\in T_+} x'_t =  \sum_{t\in T_+, t\not=\hat{t}} x'_t + x'_{\hat{t}}
  = \alpha\sum_{t\in T_+, t\not=\hat{t}} x_t + \alpha x_{\hat{t}} + 1-\alpha
  = \alpha\sum_{t\in T_+} x_t + 1-\alpha = 1.
\]
For \eqref{eq:BPineq1}, we will prove
\begin{equation}\label{TTT}
\FRAC{1}{\NORM{\sum_{t\in T_+}\tilde{\ba}_t x'_t}^2} \geq  \FRAC{1}{\NORM{\sum_{t\in T_+}\tilde{\ba}_t x_t}^2} + \FRAC{1}{\NORM{\tilde{\ba}_{\hat t}}^2}.
\end{equation}
Then (\ref{eq:BPineq1}) readily follows since $\NORM{\tilde{\ba}_{\hat{t}}}=1$. 
The inequality \eqref{TTT} is obvious from the following lemma.

\begin{lemma} \label{lem:minlength}
Let $\ba, \bz \in \R^m$ such that $\ba^T\bz \leq 0$, and consider 
to find the minimum length point between them; 
\[
 \left\{
  \begin{array}{ll}
   \mbox{\rm minimize} & f(\alpha) = \NORM{\alpha \bz + (1-\alpha) \ba}^2 \\
   \mbox{\rm subject to} & 0\leq \alpha \leq 1.
  \end{array}
 \right.
\]
The solution $\alpha^\ast$ of the above problem is
\[
\alpha^\ast = \FRAC{\ba^T(\ba-\bz)}{\NORM{\ba-\bz}},
\]
and it holds that 
\[
 f(\alpha^\ast) \leq \FRAC{\NORM{\ba}^2\NORM{\bz}^2}{\NORM{\ba}^2 + \NORM{\bz}^2}.
\]
\end{lemma}
\begin{proof}
Notice that $\alpha^\ast$ is the solution of $df/d\alpha = 0$.
It is easy to see that $0\leq \alpha^\ast \leq 1$ if and only if $\ba^T\bz\leq 0$.
Finally, we have:
\begin{eqnarray*}
f(\alpha^\ast) &=& \NORM{\ba}^2 - \FRAC{(\ba^T(\ba-\bz))^2}{\NORM{\ba-\bz}^2} 
 = \NORM{\ba}^2 - \FRAC{\NORM{\ba}^4 - 2 \ba^T\bz \NORM{\ba}^2 + (\ba^T\bz)^2}{\NORM{\ba}^2 - 2\ba^T\bz + \NORM{\bz}^2} \\
 &=& \FRAC{\NORM{\ba}^2\NORM{\bz}^2 - (\ba^T\bz)^2}{\NORM{\ba}^2+\NORM{\bz}^2 - 2\ba^T\bz} 
 \leq  \FRAC{\NORM{\ba}^2\NORM{\bz}^2}{\NORM{\ba}^2 + \NORM{\bz}^2}.
\end{eqnarray*}
\end{proof}

\bibliographystyle{abbrvurl}
\bibliography{bib_plain}
\end{document}